\documentclass[letterpaper, 11pt]{article}
\usepackage[left=1in, right=1in, top=1in, bottom=1in]{geometry}
\usepackage[T1]{fontenc}
\usepackage{amsmath, amsthm, amssymb, mathtools}
\usepackage{graphicx}
\usepackage{xcolor}
\usepackage{algorithm}
\usepackage{algpseudocode}
\usepackage[colorlinks,linkcolor=blue,citecolor=blue,urlcolor=magenta,linktocpage,plainpages=false]{hyperref}
\usepackage{caption}
\usepackage{subcaption}
\usepackage{setspace}
\usepackage{authblk}
\usepackage{changepage}
\usepackage{tikz}
\usepackage{enumitem}
\usepackage{array}
\usepackage[edges]{forest}
\newcolumntype{x}[1]{>{\centering\arraybackslash\hspace{0pt}}p{#1}}

\newtheorem{theorem}{Theorem}

\newtheorem{remark}{Remark}

\newtheorem{definition}{Definition}

\DeclarePairedDelimiter\floor{\lfloor}{\rfloor}

\graphicspath{{images/}}

\title{Non-Monotonicity of Branching Rules with respect to Linear Relaxations}
\author[1]{Prachi Shah\thanks{prachi.shah@gatech.edu}}
\author[1]{Santanu S. Dey\thanks{santanu.dey@isye.gatech.edu}}
\author[2]{Marco Molinaro\thanks{mmolinaro@microsoft.com}}
\affil[1]{School of Industrial and Systems Engineering, Georgia Institute of Technology}
\affil[2]{Microsoft Research (Redmond) and Computer Science Department, PUC-Rio}
\date{\vspace{-5ex}}

\begin{document}

%

%
%
%
%
%
\maketitle              

\begin{abstract}
    Modern mixed-integer programming solvers use the branch-and-cut framework, where cutting planes are added to improve the tightness of the linear programming (LP) relaxation, with the expectation that the tighter formulation would produce smaller branch-and-bound trees.
    In this work, we consider the question of whether adding cuts will always lead to smaller trees for a given fixed branching rule. We formally call such a property of a branching rule monotonicity.  
    We prove that any branching rule which exclusively branches on fractional variables in the LP solution is non-monotonic
    Moreover, we present a family of instances where adding a single cut leads to an exponential increase in the size of full strong branching trees, despite improving the LP bound. 
    Finally, we empirically attempt to estimate the prevalence of non-monotonicity in practice while using full strong branching. We consider randomly generated multi-dimensional knapsacks tightened by cover cuts as well as instances from the MIPLIB 2017 benchmark set for the computational experiments. Our main insight from these experiments is that if the gap closed by cuts is small, change in
    tree size is difficult to predict, and often increases, possibly due to inherent non-monotonicity. However, when a sufficiently large gap is closed, a significant decrease in tree size may be expected.

\end{abstract}

\section{Introduction}

State-of-the-art mixed integer programming (MIP) solvers are based on the branch-and-cut framework. Classically, the cutting-plane method~\cite{dantzig1954solution,gomoryoutline} and the branch-and-bound method~\cite{land60a} were mostly explored independently. It was later discovered that combining them, by first adding cutting-planes to improve the formulation followed by applying branch-and-bound method produces good computational benefits. See~\cite{cook2010fifty,lodi2010mixed} for great expositions on the origins of the branch-and-cut framework, including discussion on landmark papers such as~\cite{crowder1983solving,van1987solving,padberg1991branch,balas1996mixed,balas1996gomory}.  In recent years, there have been multiple studies understanding the trade-off between branching and the use of cutting-planes~\cite{basu2021complexity,basu2023complexity,kazachkov2023abstract}. In this paper, we study another aspect of the interaction between cutting-planes and the branch-and-bound procedure.

Let us consider two formulations of a given MIP, where the second formulation is tighter than the first, that is the feasible region of the second Linear Programming (LP) relaxation is contained in that of the first relaxation. One could obtain the second relaxation, for example, by adding cutting-planes to the first relaxation. Suppose we have a branch-and-bound tree that solves the MIP using the first formulation. Then, it is clear that the same branch-and-bound tree, that is a tree with the same branching decisions, solves the MIP using the second formulation. 
Therefore, the branch-and-bound tree for the first formulation acts as a \emph{branch-and-bound certificate} for the second formulation as well. 
Thus, one expects that the branch-and-bound tree needed to solve the MIP using the tighter formulation becomes smaller-- this is the basis of the branch-and-cut framework. 
However, in order to use the branch-and-bound method, one must specify two rules: the node selection rule and the branching (variable selection) rule. 
Once these rules are specified, we formally obtain a \emph{branch-and-bound algorithm}. The branching rule typically uses local node LP information to decide on which variable to branch. Therefore, it may be possible that a given branch-and-bound algorithm does not produce the same trees as certificates for the first and the second formulation - even worse it may produce a larger tree for the tighter formulation. The goal of this paper is to study the effect of branching rules on the size of branch-and-bound trees with respect to different formulations. 


\paragraph{Formalizing the size of a tree for a given branching rule.}
Consider two formulations of a MIP: $P, P' \subseteq \mathbb{R}^n$ such that $P \cap \mathcal{X} = P' \cap \mathcal{X}$ where $\mathcal{X}$ represents integrality constraints. 
In order to meaningfully compare the size of branch-and-bound trees generated by a given branching rule for these two polytopes, we should control for all other sources of variability of the branch-and-bound tree:

(1.) \emph{Dual degeneracy:} If there are multiple optimal solutions of the LP at any node of the branch-and-bound tree, then we cannot control which optimal vertex is reported by the LP solver. Since branching rules often use local information, such as the list of fractional variables at a node, this can lead to different trees. An extreme case is when one of the optimal vertex is integral, but other vertices are not integral, see discussion in \cite{dey2023theoretical}. One way to resolve this situation is to compare branch-and-bound trees for $P$ and $P'$, only for objective functions that do not cause any dual degeneracy at any node of any branch-and-bound tree for both $P$ and $P'$. Let $C(P,P') \subseteq \mathbb{R}^n$ be the set of these objectives with no dual degeneracy. Note that $C(P,P')$ only discards finitely many objective functions. 

(2.) \emph{Node selection rule:} We will assume that nodes are processed in the best bound first order.  Note that once we solve a problem with any node selection, one can always recover a subtree that solves the same instance, where the node selection rule for the subtree is the best-bound rule -- that is the best-bound rule for node selection produces the smallest trees; see~\cite{dey2021branch} for properties of this rule.

We let $T^r(P, c)$ be the branch-and-bound tree produced by the branching rule $r$ when provided with relaxation $P$ and objective function $c \in C(P, P')$ with the above assumptions. We define the size of the tree as the total number of its nodes and denote it as $|T^r(P, c)|$. 
%
%
\paragraph{Our contributions.}
Given two linear relaxation $P$ and $P'$ for a MIP, that is 
$P' \cap  \mathcal{X} = P \cap \mathcal{X}$,  we say $P'$ is tighter than $P$ when $P' \subseteq P$. Motivated by the discussion above, we define the following:
\begin{definition} \label{def:monotonic}
A branching rule $r$ is said to be monotonic, if for any MIP, given two linear relaxations $P$ and $P'$ with $P'$ being tighter than $P$, for all $c \in C(P, P')$ we have that $$|T^r(P', c)| \le |T^r(P, c)|$$.
\end{definition}
Ideally, we may hope for ``good'' branching rules to be monotonic, which would be a theoretical justification of the branch-and-cut algorithm. 

Our contributions can be summarized as follows: 
\begin{enumerate}[topsep=0pt]
    \item 
    We show that any rule that branches exclusively on fractional variables in the LP solution is non-monotonic. As a consequence, most standard branching rules in literature, including full strong branching (FSB), are not monotonic.
    \item In particular, we prove that adding a single cut may lead to an exponential increase in the size of branch-and-bound trees for most of these branching rules.
    \item Through computational experiments we attempt to estimate the prevalence of non-monotonicity in practice while using FSB. We do so by applying cover cuts on randomly generated multi-dimensional knapsacks as well as by considering cuts applied by SCIP \cite{BestuzhevaEtal2021OO} on MIPLIB 2017 benchmark set \cite{miplib2017}. Our main insight from these experiments is that if the gap closed by cuts
is small, change in tree size is difficult to predict, and often increases, possibly due to inherent non-monotonicity. However, when a sufficiently large gap is closed, a significant decrease in tree size may be expected. 
\end{enumerate}
The rest of the paper is organized as follows. In Section \ref{sec:theorectical_proofs} we provide the background on branching rules in literature and present our theoretical contributions. Section \ref{sec:computational} discusses the computational experiments and the results. Finally in Section \ref{sec:conclusion} we conclude and discuss future directions of research.

\section{Non-Monotonicity of Branching Rules} \label{sec:theorectical_proofs}

\subsection{Background on branching rules} \label{sec:branchingrules}
One of the most important branching rules is called the \textit{full strong branching rule} (FSB)~\cite{applegate1995finding}. It is empirically known to produce one of the smallest trees as compared to all other branching rules~\cite{achterberg2005branching}. 
Let $\Delta^0_i$ and $\Delta^1_i$ be the LP gains, i.e. the difference in the LP value of a node and its children when tentatively branched on variable $i$. Then $\Delta^0_i$ and $\Delta^1_i$ are aggregated to obtain a score for variable $i$ and the variable with the highest score is selected for branching\footnote{If a child node is infeasible, the corresponding LP gain is infinite. If multiple branching candidates lead to infeasible children, we branch on the variable with two infeasible nodes if such exists, else, maximize the gain on the feasible branch.}. 
The following functions are the most commonly used \cite{achterberg2005branching,achterberg2007constraint,linderoth1999computational} due to their robust performances,
\begin{enumerate}[topsep=1ex, itemsep=0ex,partopsep=0ex,parsep=1ex]
    \item Product rule : score($i$) = $\max\{\Delta^0_i, \epsilon\} \cdot \max\{\Delta^1_i, \epsilon\}$, where $\epsilon > 0$ is small 
    \item Linear rule : score($i$) = $(1 - \mu) \min\{\Delta^0_i, \Delta^1_i\} + \mu \max\{\Delta^0_i, \Delta^1_i\}$ for some $\mu \in [0, 1]$. In this paper, we choose $\mu = 1/6$ following the recommendation in~\cite{achterberg2005branching}.
    \item Le Bodic-Nemhauser ratio rule~\cite{le2017abstract}: score($i$) = $1/\phi^*$, where $\phi^*$ is the unique root greater than 1 of the 
    polynomial, $p(\phi) = \phi^{\max\{\Delta^0_i, \Delta^1_i\}} - \phi^{|\Delta^0_i - \Delta^1_i|} - 1 = 0$. 
\end{enumerate}

\subsection{Any rule that branches exclusively on fractional variables is non-monotonic}

Most standard branching rules branch including strong branching, reliability branching, most fractional branching, and maximum separating distance branching 
define the set of candidate branching variables as those that have a fractional value in the LP optimal solution, whereas variables that have an integer value are eliminated from consideration or assigned a zero score.
For our first theoretical result, we show that any branching rule which branches on fractional variables exclusively is non-monotonic, implying that all the above-mentioned rules are also non-monotonic in general.

\begin{theorem}\label{theorem:frac_rules}
    Any branching rule that branches only on fractional variables in the optimal LP solution is non-monotonic.
\end{theorem}
\begin{proof}
    Consider the polytope $Q$ formed by the convex hull of the union of the following sets of vertices, 
    \begin{align*}
        V_1 = \{ x \in \mathbb{R}^4 : x_1 = 1/2, \: (x_2, x_3) \in \{0, 1\}^2, \: x_4 = 0 \} \\
        V_2 = \{ x \in \mathbb{R}^4 : x_2 = 1/2, \: (x_1, x_3) \in \{0, 1\}^2, \: x_4 = 1 \} \\
        V_3 = \{ x \in \mathbb{R}^4 : x_3 = 1/2, \: (x_1, x_2) \in \{0, 1\}^2, \: x_4 = 1 \} 
    \end{align*}
    Polytope Q is closely related to cross-polytope considered \cite{bodur2017cutting}.
    Further consider the polytope, $P$ created by introducing an additional vertex $ v_0 = (1, 1, 1, 1/2)$, that is, $P = \text{conv} (Q \cup v_0)$ and therefore it trivially follows that $Q \subset P$. In fact, $Q$ can be obtained from $P$ as the lift-and-project closure of $P$ with respect to $x_4$.

    Now, consider the two pure binary programs maximizing the vector $c = (1, 1, 1, 1-\epsilon)$ over $P$ and $Q$ for some $\epsilon \in (0, 1)$. Both problems are infeasible since $P$ and $Q$ do not contain any integral points, but the formulation with $Q$ is tighter. Consider any branching rule that considers only the variables that are fractional in the optimal LP solution. We will now show that it generates a larger tree to prove infeasibility for the formulation with $Q$ than that with $P$.
    
    \begin{figure}
    \captionsetup[subfigure]{aboveskip=12pt,belowskip=0pt}
      \centering
    
      \begin{subfigure}{0.8\textwidth}
        \centering
        \begin{tikzpicture}[level distance=2cm,
                            level 1/.style={sibling distance=6.4cm},
                            level 2/.style={sibling distance=3.3cm},
                            level 3/.style={sibling distance=1.6cm},
                            treenode/.style={draw, shape=rectangle, minimum width = 1.1cm, minimum height = 0.7cm},
                            edge from parent/.append style={font=\footnotesize}]
          \draw[black!0] (1, -0.5) rectangle (1,-4);  
          \node [treenode] {$V_1 \cup V_2 \cup V3 $}
            child {node [treenode] (a) {$V_1$}  
              child {node [treenode] {$\emptyset$} edge from parent node [left] {$x_1=0$}}
              child {node [treenode] {$\emptyset$} edge from parent node [right] {$x_1=1$}}
              edge from parent node [left] {$x_4=0$} 
            }
            child {node [treenode] {$V_2 \cup V_3$}
              child {node [treenode] {$V_3 \cap \{x_2=0\}$} 
                child {node [treenode] {$\emptyset$} edge from parent node [left] {$x_3=0$}}
                child {node [treenode] {$\emptyset$} edge from parent node [right] {$x_3=1$}}
              edge from parent node [left] {$x_2=0$} }
              child {node [treenode] {$V_3 \cap \{x_2=1\}$} 
                child {node [treenode] {$\emptyset$} edge from parent node [left] {$x_3=0$}}
                child {node [treenode] {$\emptyset$} edge from parent node [right] {$x_3=1$}}
              edge from parent node [right] {$x_2=1$} }
              edge from parent node [right] {$x_4=1$} 
            };
          
        \end{tikzpicture}
        \caption{Tree given formulation $P$ for any rule considering only fractional variables} \label{fig:fracP}
        \vspace{36pt}
      \end{subfigure}
      \begin{subfigure}{0.8\textwidth}
        \centering
        \begin{tikzpicture}[level distance=2cm,
                            level 1/.style={sibling distance=6.4cm},
                            level 2/.style={sibling distance=3.3cm},
                            level 3/.style={sibling distance=1.6cm},
                            treenode/.style={draw, shape=rectangle, minimum width = 1.1cm, minimum height = 0.7cm},
                            edge from parent/.append style={font=\footnotesize}]
          \draw[black!0] (1, -0.5) rectangle (1,-4);  
          \node [treenode] {$V_1 \cup V_2 \cup V3 $}
            child {node [treenode] {$(V_2 \cup V_3) \cap \{x_1=0\}$}
              child {node [treenode] {$(0, 0, 1/2)$} 
                child {node [treenode] {$\emptyset$} edge from parent node [right] {$\: x_3$}}
                child {node [treenode] {$\emptyset$} edge from parent node [right] {}}
              edge from parent node [left] {$x_2=0$} }
              child {node [treenode] {$(0, 1, 1/2)$} 
                child {node [treenode] {$\emptyset$} edge from parent node [right] {$\: x_3$}}
                child {node [treenode] {$\emptyset$} edge from parent node [right] {}}
              edge from parent node [right] {$x_2=1$} }
              edge from parent node [right] {$x_1=0$} 
            }
            child {node [treenode] {$(V_2 \cup V_3) \cap \{x_1=1\}$}
              child {node [treenode] {$(1, 0, 1/2)$} 
                child {node [treenode] {$\emptyset$} edge from parent node [right] {$\: x_3$}}
                child {node [treenode] {$\emptyset$} edge from parent node [right] {}}
              edge from parent node [left] {$x_2=0$} }
              child {node [treenode] {$(1, 1, 1/2)$} 
                child {node [treenode] {$\emptyset$} edge from parent node [right] {$\: x_3$}}
                child {node [treenode] {$\emptyset$} edge from parent node [right] {}}
              edge from parent node [right] {$x_2=1$} }
              edge from parent node [right] {$x_1=1$} 
            };
          
        \end{tikzpicture}
        \caption{Tree given formulation $Q$ for any rule considering only fractional variables } \label{fig:fracQ}
      \end{subfigure}
    
      \caption{Illustration of branch-and-bound trees for the example in Theorem \ref{theorem:frac_rules}.} \label{fig:frac_proof}
      \vspace{-12pt}
    \end{figure}
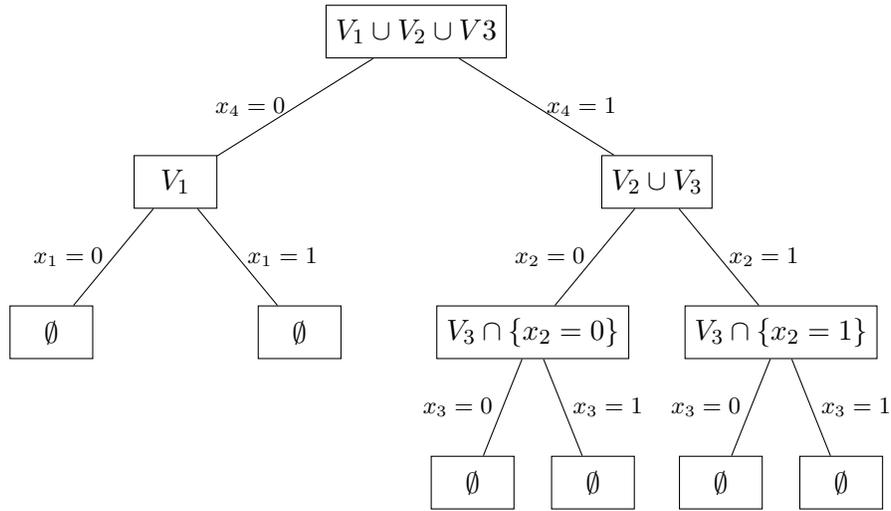
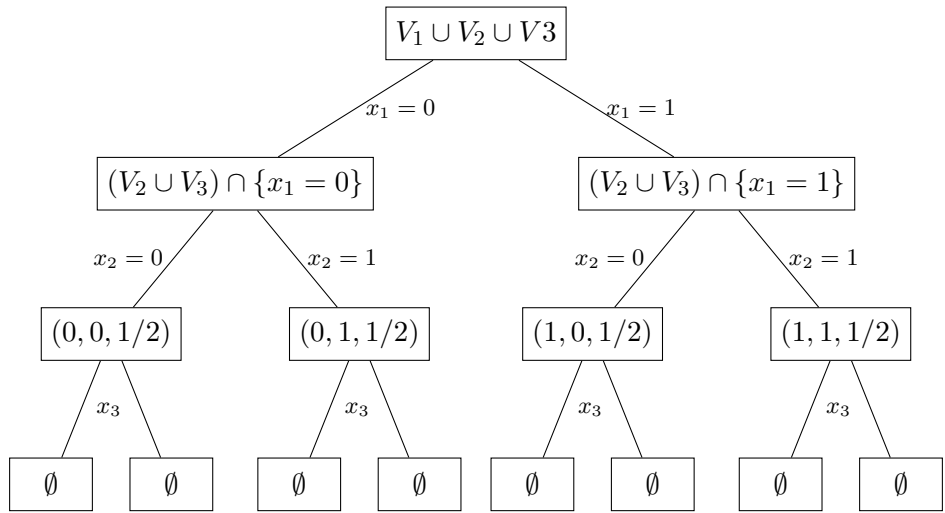

    First, consider the formulation with $P$. The LP optimal solution at the root node is the vertex $v_0$, where only $x_4$ is fractional. So the branching rule must branch on $x_4$. On the branch where $x_4 = 0$, at any vertex, only $x_1$ is fractional and branching on it leads to infeasible nodes on both branches. Whereas, on the branch with $x_4 = 1$, $x_1$ is integral on all vertices and therefore isn't considered on branching. To prove infeasibility, any branching rule must branch on both $x_2$ and $x_3$ as illustrated in Fig.~\ref{fig:fracP} leading to a tree with 11 nodes in total. 

    On the other hand, in the case of the formulation with $Q$, $x_4$ is integral on all vertices of the polytope and is never considered for branching. The branch-and-bound tree is, therefore, the same as that for the cross-polytope which needs to branch on all of the 3 variables, generating a complete tree in $x_1, x_2, x_3$ \cite{dadush2020complexity, dey2023theoretical} shown in Fig.~\ref{fig:fracQ}. This tree has 15 nodes, and the rule is therefore non-monotonic. 

\end{proof}


\subsection{Exponential increase in size of branch-and-bound tree}
We next focus on full strong branching and show that not only it is non-monotonic but in the worst case it can generate exponentially larger trees when provided with a tighter formulation, for any of the scores discussed in Section \ref{sec:branchingrules}.

\begin{figure}[t]
\begin{tabular}{>{\centering\arraybackslash}b{\dimexpr0.4\linewidth-2\tabcolsep\relax} >{\centering\arraybackslash}b{\dimexpr0.6\linewidth-2\tabcolsep\relax}}

\scalebox{0.65}{
  \begin{tikzpicture}[xscale=5, yscale=5]
    \draw (0,0) rectangle (1,1);
  
    \filldraw[fill=blue!20!white] (0.9,0.5) node[above right] {A} -- (1,0.35) node[right] {B} -- (1,0) node[below] {C} -- (0,0) node[below right] {F} -- (0,0.3) node[left] {E} -- (0.1,1) node[above] {D} -- cycle;

    \draw[dashed] (0.5,0.75) node[above] {A'} -- (1,0.1) node[right] {B'};
    
    \draw [->] (-0.1,0) -- (1.1,0);
    \node [below] at (1.1,0) {$x$};
  
    \draw [->] (0,-0.1) -- (0,1.1);
    \node [left] at (0,1.1) {$y$};
    
  \end{tikzpicture}
  }
\caption{Polytopes $P$ and $P'$}\label{fig:sb_example}
    &
\renewcommand{\arraystretch}{1.05}
    \footnotesize
    \begin{tabular}{|x{1.4cm}|x{1.4cm}|x{1.4cm}|x{1.4cm}|}\hline
      Vertex&  $ x $ & $ y $ & Obj val\\ \hline
        $A$ & 0.9 & 0.5 & 7.9 \\
        $B$ & 1 & 0.35 & 7.75 \\
        $\phantom{'}A'$ & 0.5 & 0.75 & 6.75 \\
        $\phantom{'}B'$ & 1 & 0.1 &  6.5 \\
        $C$ & 1 & 0 & 6 \\
        $D$ & 0.1 & 1 & 5.6 \\
        $E$ & 0 & 0.3 & 1.5 \\
        $F$ & 0 & 0 & 0  \\\hline
      \end{tabular}
\captionof{table}{Vertices of the polytopes}\label{tab:vertices}
\end{tabular}
\vspace{-18pt}
\end{figure}

\begin{theorem}\label{theorem:exp_SB-P}
Adding a single cut can increase the size of the branch-and-bound tree exponentially when the tree is based on full strong branching with any one of the following scores: product score, linear score and ratio score. 
\end{theorem}

\begin{proof}
We prove this result for full strong branching with product score (FSB-P). The proof is identical for all the other scores. 
Consider the following polytopes,
\begin{align}
P &= \big\{(x, y) \in [0, 1]^2 \ |  -7x + y \le 0.3, \ 5x + 8y \le 8.5, 3x + 2y \le 3.7 \big\},    \\
P' &= \big\{(x, y) \in P \ | \ 13x + 10y \le 14 \big\},
\end{align}
where $P \cap \mathbb{Z}^2 = P' \cap \mathbb{Z}^2.$ The polytope $P$ is illustrated in Fig.~\ref{fig:sb_example} as $ABCFED$ and the polytope $P'$ as $A'B'CFED$. Observe that $P' \subset P$, 
is obtained by adding a single inequality to $P$ that is valid for all integer points in $P$. 
Now, consider the 2-dimensional MIP with formulations, 
\begin{align}
   \max_{x, y} \ \{ 6x + 5y \ | \ (x, y) \in P \cap \mathbb{Z}^2 \} \label{mip:SB_P} \\
   \max_{x, y} \ \{ 6x + 5y \ | \ (x, y) \in P' \cap \mathbb{Z}^2 \}  \label{mip:SB_P'}
\end{align} 
and note the formulation with $P'$ is tighter.
Neither of these formulations has dual degeneracy. The objective values of all vertices of $P$ and $P'$ are presented in Table~\ref{tab:vertices}.
Observe that $A$ is the optimal solution over the LP relaxation $P$ and $C$ is the optimal solution of the MIP.

Next, consider the following MIP, referred to as $Q_{n}$
\begin{subequations}
\begin{alignat}{3}
    \max \quad &6 \sum_{i=1}^n x_i + 5\sum_{i=1}^n y_i \\
    \text{s.t.} \quad &(x_i, y_i) \in P \cap \mathbb{Z}^2 &\text{ for } i = 1, \hdots, n\\
    & 13 x_i + 10y_i \le z &\text{ for } i = 1, \hdots, n \label{eq:exp_con2}\\
    & z \le 16.7 \label{eq:exp_con3}, \ z \in \mathbb{R}. 
\end{alignat}
\end{subequations}

Notice that since $z$ does not appear in the objective, and is only present in constraints (\ref{eq:exp_con2}) and (\ref{eq:exp_con3}), these constraints imply $13x_i + 10y_i \le 16.7$ for all $i \in [n]$ and are therefore redundant. 
Thus, the problem can be decomposed into $n$ independent MIPs where each subproblem is (\ref{mip:SB_P}). Consequently, the optimal solution to the LP relaxation is $(x_i, y_i) = (0.9, 0.5)$, i.e. vertex $A$ for all $i \in [n]$. 

Let $T(n)$ be the size of the tree for solving $Q_{n}$ using FSB-P rule.

At the root node, by symmetry and the independence of subproblems, the LP gains and consequently the FSB-P scores of all variables are identical to their counterparts in MIP (\ref{mip:SB_P}). Consider tentatively branching on $x$ at the root node of the two-dimensional MIP using FSB-P. 
The optimal LP solution for the face $x = 0$ is $E$, and that of face $x = 1$ is $B$. Similarly, when tentatively branching on $y$, the optimal solution of the LP relaxation on face $y = 0$ is $C$, and that on face $y = 1$ is $D$. The scores are then calculated as, 
\begin{align*}
    \text{score}(x) &= \Delta^0_x \cdot \Delta^1_x = (z_A - z_E) \cdot (z_A - z_B) = 0.96. \\
    \text{score}(y) &= \Delta^0_y \cdot \Delta^1_y = (z_A - z_C) \cdot (z_A - z_D) = 4.37.
\end{align*}

Therefore, at the root node of $Q_n$, $\text{score}(x_i) = 0.96$ and $\text{score}(y_i) = 4.37$ for all $i \in [n]$. 
Ties may be broken arbitrarily and without loss of generality, FSB-P rule branches on $y_n$. 

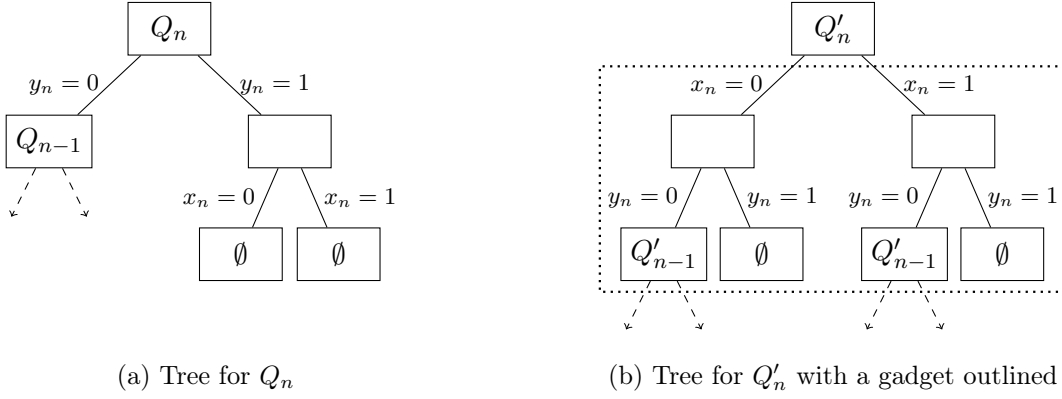
\begin{figure}
\captionsetup[subfigure]{aboveskip=12pt,belowskip=0pt}
  \centering

  \begin{subfigure}{0.47\textwidth}
    \centering
    \begin{tikzpicture}[level distance=1.5cm,
                        level 1/.style={sibling distance=3.2cm},
                        level 2/.style={sibling distance=1.3cm},
                        treenode/.style={draw, shape=rectangle, minimum width = 1.1cm, minimum height = 0.7cm},
                        edge from parent/.append style={font=\footnotesize}]
      \draw[black!0] (1, -0.5) rectangle (1,-4);  
      \node [treenode] {$Q_n$}
        child {node [treenode] (a) {$Q_{n-1}$}  
          edge from parent node [left] {$y_n=0$} 
        }
        child {node [treenode] {$\phantom{3}$}
          child {node [treenode] {$\emptyset$} edge from parent node [left] {$x_n=0$} }
          child {node [treenode] {$\emptyset$} edge from parent node [right] {$x_n=1$} }
          edge from parent node [right] {$y_n=1$} 
        };
        \path[->] (a) edge[dashed] +(-0.5,-1);
        \path[->] (a) edge[dashed] +(+0.5,-1);
      
    \end{tikzpicture}
    \caption{Tree for $Q_n$ } \label{fig:TQn}
  \end{subfigure}
  \hfill
  \begin{subfigure}{0.52\textwidth}
    \centering
    \begin{tikzpicture}[level distance=1.5cm,
                        level 1/.style={sibling distance=3.2cm},
                        level 2/.style={sibling distance=1.3cm},
                        treenode/.style={draw, shape=rectangle, minimum width = 1.1cm, minimum height = 0.7cm},
                        edge from parent/.append style={font=\footnotesize}]
      \node [treenode] {$Q'_n$}
        child {node [treenode] {$\phantom{3}$}
          child {node [treenode] (b) {$Q'_{n-1}$} edge from parent node [left] {$y_n=0$} }
          child {node [treenode] {$\emptyset$} edge from parent node [right] {$y_n=1$} }
          edge from parent node [left] {$x_n=0$} 
        }
        child {node [treenode] {$\phantom{3}$}
          child {node [treenode] (a) {$Q'_{n-1}$} edge from parent node [left] {$y_n=0$} }
          child {node [treenode] {$\emptyset$} edge from parent node [right] {$y_n=1$} }
          edge from parent node [right] {$x_n=1$} 
        };
        \path[->] (a) edge[dashed] +(-0.5,-1);
        \path[->] (a) edge[dashed] +(+0.5,-1);
        \path[->] (b) edge[dashed] +(-0.5,-1);
        \path[->] (b) edge[dashed] +(+0.5,-1);
        \draw[dotted, line width=0.3mm] (3.1, -0.5) rectangle (-3.1,-3.5);        
    \end{tikzpicture}
    \caption{Tree for $Q'_n$ with a gadget outlined} \label{fig:T'Qn}
  \end{subfigure}

  \caption{Illustration of strong branching trees for the example in Theorem \ref{theorem:exp_SB-P}.} \label{fig:exp_proof}
  \vspace{-12pt}
\end{figure}

Let's first consider the node $y_n = 0$.
Here, $x_n = 1$ in the LP optimal solution and therefore $x_n$ is not considered for branching any further. This holds true for any descendent node. Thus the problem at this node is equivalent to $Q_{n-1}$ and the size of the subtree rooted at 
this node is at most $T(n-1)$. 

Now consider the node where $y_n = 1$. 
If this node is not pruned by bound, then since 
$x_n$ is fractional in the LP solution at this node and branching on it generates two infeasible nodes leading to an infinite score for $x_n$. On the other hand, the scores of $x_i, y_i$ for all $i \neq n$ remain finite. 
Thus, FSB-P
chooses to branch on $x_n$ to create two infeasible nodes as depicted in Fig.~\ref{fig:TQn}. Thus, we have, 
\begin{equation}
    T(n) \le 4 + T(n-1)
\end{equation}
Applying the inequality recursively, and using $T(0) = 1$, we have $|T^{\text{FSB}}(Q_{n})| = T(n) \le 4n + 1$.

Next, consider the case where we add a single valid cut $z \le 14$.  The constraints (\ref{eq:exp_con2}) and (\ref{eq:exp_con3}), are now equivalent to a $13x_i + 10y_i \le 14$ for all $i$, which is exactly the cut that generates $P'$ from $P$. 
We refer to this problem as $Q'_{n}$ and observe that it can be decomposed into $n$ independent MIPs where each subproblem is (\ref{mip:SB_P'}). It is therefore a tighter formulation of $Q_n$. The optimal solution to the LP relaxation of $Q'_{n}$ is $(x_i, y_i) = (0.5, 0.75)$, i.e. vertex $A'$ for all $i \in [n]$, with optimal objective value $z^*_{LP} = 6.75n$. The optimal objective value of the MIP is $z^*_{IP} = 6n$ and therefore the additive integrality gap is $G = 0.75n$.


Following the same arguments as used in the analysis of $Q_n$, the FSB-P scores of all variables can be computed by considering the two-dimensional subproblems individually.
At the root node of (\ref{mip:SB_P'}), the scores for $x$ and $y$ can be similarly computed as follows, 
\begin{alignat*}{2}
    \text{score}(x) &= (z_{A'} - z_E) \cdot (z_{A'} - z_{B'}) 
    = 1.3125 \\
    \text{score}(y) &= (z_{A'} - z_C) \cdot (z_{A'} - z_{D}) 
    = 0.8625 
\end{alignat*}

Therefore, due to the symmetry and independence of subproblems, $\text{score}(x_i) = 1.3125$ and $\text{score}(y_i) = 0.8625$ for all $i \in [n]$ at the root node of $Q'_n$. Thus, 
FSB-P rule branches on $x_n$, breaking ties arbitrarily without loss of generality. 
At both of the nodes thus created by fixing $x_n$, further fixing $y_n = 1$ leads to infeasibility, giving $y_n$ an infinite score whereas the scores of variables in other subproblems remain unchanged. Thus, both of these nodes branch on $y_n$ if not pruned by bound. 
More generally, 
any node created by branching on any $x_i$, if not pruned, must branch on $y_i$.

Once both $x_n$ and $y_n$ are fixed, 
the restricted MIP is equivalent to solving $Q'_{n-1}$, allowing the branching pattern to repeat. We call this repeating block of nodes a gadget and it is indicated in  Fig.~\ref{fig:T'Qn}. 
In general, a
gadget corresponding to subproblem $i$ branches on $x_i$ and $y_i$ in sequence and includes $2$ internal nodes in addition to the following $4$ nodes in the lowest level: (i) $(x_i, y_i) = (0, 0)$ with LP gain $z_{A'} - z_F = 6.75$, (ii) $(x_i, y_i) = (0, 1)$ which is infeasible, (iii) $(1, 0)$ with LP gain $z_{A'} - z_C = 0.75$, and (iv) 
$(x_i, y_i) = (1, 1)$ which is infeasible.

As the tree is composed of repetitions of possibly partial gadgets, we will estimate the size of the tree by counting the number of complete gadgets. 
It is straightforward to see that a gadget is incomplete only due to pruning by bound.  
A node may be pruned by bound only if the gap closed by the tree at the node exceeds the integrality gap of $G$. Consider any feasible node at depth $2d$, where the root node is at depth 0 by convention. The maximum LP gain by a gadget at any feasible node is $g = 6.75$. Therefore, the gap closed by the tree at any feasible node at depth $2d$ is at most $gd$. This implies that at any feasible node till depth $2\, \floor{G/g}$, the gap closed is no more than $G$ and there is no pruning by bound till this depth.  

Using the fact that the number of gadgets doubles with an increase in depth by 2, the number of complete gadgets in the FSB-P tree is at least $2^{\floor{G/g}} - 1$.
Thus, the number of nodes in the tree is at least, 
\begin{align*}
    |T^{\text{FSB}}(Q'_{n})| & \ \ge 
    \ 6 \ (2^{\floor{G/g}} - 1 )\ = 
    \ 6 \ (2^{\floor{\frac{0.75n }{6.75}}} - 1 ).\
\end{align*}
\end{proof}

\begin{remark}
The above result holds even if we do not assume the best-bound node selection rule, since in estimating an upper bound on $|T^{\text{FSB}}(Q_{n})|$ we did not consider pruning by bound.
\end{remark}


\section{Computational Experiments} \label{sec:computational}
In this section, our goal is to conduct computational experiments to ascertain how common is it for the strong branching rule to exhibit non-monotonicity. We focus our attention on the product rule for FSB (denoted as FSB-P) since it has been shown to be computationally superior \cite{achterberg2007constraint} to many other rules. 

\subsection{Cover Cuts for Multi-dimensional Knapsack Problem (MKP)}\label{sec:MKP}
\paragraph{Instances.}
Consider an instance $\mathcal{I}$ of the multi-dimensional knapsack problem (MKP), which can be formulated as follows,
\begin{eqnarray}
\label{mip:multiknapsack}
  \max\Bigg\{\sum_{i=1}^n c_i x_i \ \bigg| \ \sum_{i=1}^n a^j_i x_i \le b^j, \forall j \in [m], \ x \in \{0, 1\}^n\Bigg\}.
\end{eqnarray}
We randomly generated $100$ instances with $20$ variables and $50$ constraints. 
Weights $a_i^j$ are independent and set to $0$ with probability $0.25$, else uniformly picked from the set $\{1, \hdots 1000\}$. Capacity $b^j$ is set to 
$\lfloor 0.5\cdot \sum_{i = 1}^n a^j_i\rfloor$. 
Prices $c_i$ are independently picked from the uniform distribution on the interval $[0, 1)$. By randomly sampling objective coefficients from a continuous distribution, we ensure that almost surely the instances do not exhibit dual degeneracy.

\paragraph{Experiments.}
For each knapsack constraint in the instance, we try to find cover cuts
separating the root LP solution. 
We solve the following separation problem, 
\begin{alignat}{3} \label{CGIP}
   d_j^* = \max \quad & \sum_{i=1}^n x^*_i w_i - \sum_{i=1}^n w_i + 1 \notag \\
    \text{s.t.} & \sum_{i=1}^n a^j_i w_i \ge b^j + 1 \quad \\
    &w_i \in \{0, 1\}  &  \text{for all } i = 1, \hdots, n \notag
\end{alignat}
If the above cut-generating IP is infeasible, there are no cover cuts for the knapsack constraint $j$ and the constraint is in fact redundant. If the cut-generating IP is feasible, let $\hat{w}^j$ be the optimal solution and set $C^j$ be the support of $\hat{w}^j$. Then a valid cover cut is given by,
\begin{equation}
    \sum_{i \in C^j} x_i \ \leq \ |\: C^j| - 1
\end{equation}
Moreover, the cut separates $x^*$ from the integer hull of the MKP if and only if $d_j^* > 0$. 
Let $\mathcal{C}_\mathcal{I}$ be the set of all violated cover inequalities discovered. 
For the first set of experiments, we add each cover cut individually. That is for every $C \in \mathcal{C}_\mathcal{I}$, we solve, 
\begin{eqnarray*} \label{model:onecut}
   \max \Bigg\{\sum_{i=1}^n c_i x_i \ \bigg| \ \sum_{i=1}^n a^j_i x_i \le b^j \ \forall j \in [m], \ \sum_{i \in C} x_i \ \leq \ |\: C \:| - 1, x\in \{0, 1\}^n\Bigg\},
\end{eqnarray*}
using branch-and-bound with FSB-P
rule.
 For the next experiment, we add all separating cover cuts found for the instance. That is for every $\mathcal{I}$, we solve, 
\begin{eqnarray*}
 \label{model:allcut}
   \max \Bigg\{\sum_{i=1}^n c_i x_i \ \bigg| \ \sum_{i=1}^n a^j_i x_i \le b^j \forall j \in [m], \sum_{i \in C} x_i  \leq  |\: C \:| - 1 \ \forall C \in \mathcal{C}_\mathcal{I}, 
   x \in \{0, 1\}^n \Bigg\}.
\end{eqnarray*}

Let the LP relaxation of the original formulation have objective value $z$ and let the size of strong branching tree for this formulation be $T$. Define $\hat{z}$ and $\hat{T}$ similarly for a formulation strengthened by cuts. Moreover, let $z_{IP}$ be the value of the integer optimal solution. We compute the gap closed by the cuts $\Delta G$ and the change in tree size $\Delta T$ as follows, 
\begin{equation}
    \Delta G = \frac{z - \hat{z}}{z - z_{IP}}, \qquad \Delta T = \frac{\hat{T} - T}{T}.  \label{eq:delta_T_G}
\end{equation}
Lastly, the depth of cut, $\Delta d$ is the Euclidean distance between the LP solution and the separating hyperplane.

The experiments were conducted on a Python 3.10-based implementation of the naive branch-and-bound framework to ensure a controlled experiment, unaffected by other sophisticated processes of modern solvers. Gurobi 10.0.1 was used as the LP solver. 
FSB-P was implemented purely as an oracle returning a branching variable, without influencing any other process within branch-and-bound, thus allowing fair node counting \cite{gamrath2018measuring}.
Nodes were processed in the best-bound first order. Finally, neither presolve, nor additional cutting planes besides the cover cuts mentioned above were employed.

\vspace{-6pt}
\begin{figure}[ht]
	\centering
	\begin{subfigure}[c]{0.32\textwidth}
		\centering
		\includegraphics[width=\textwidth]{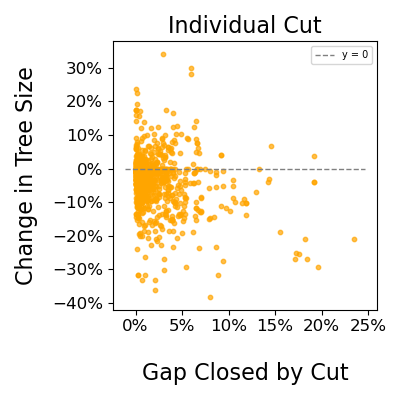}
		\caption{$\Delta T$ versus $\Delta G$ when an individual cut is added.  } \label{20:gap}
	\end{subfigure}
        \hfill
	\begin{subfigure}[c]{0.32\textwidth}
		\centering
		\includegraphics[width=\textwidth]{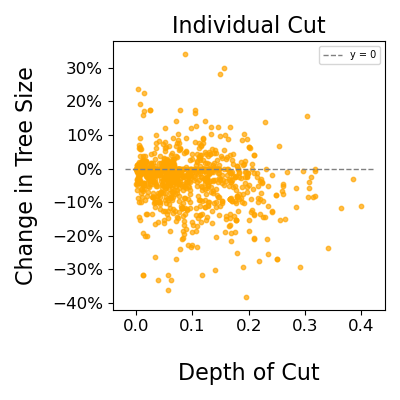}
		\caption{$\Delta T$ versus $\Delta d$ when an individual cut is added.  } \label{20:distance}
	\end{subfigure}
        \hfill
	\begin{subfigure}[c]{0.32\textwidth}
		\centering
		\includegraphics[width=\textwidth]{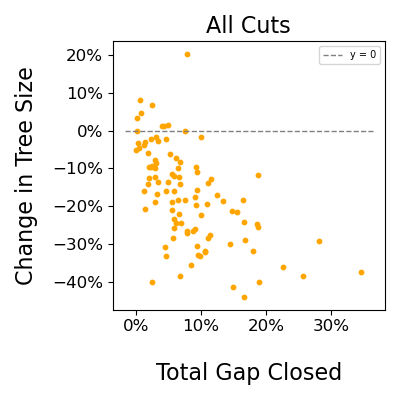}
        \caption{$\Delta T$ versus $\Delta G$ when all cuts are added. } \label{20:all}
    	\end{subfigure}
	\caption{Impact of adding cover cuts on the size of FSB-P trees for MKP instances.}
	\label{fig:mkp}
 \vspace{-12pt}
\end{figure} 



\paragraph{Results.} The results of the experiments are plotted in Fig.~\ref{fig:mkp}. For the first experiment, Fig.~\ref{20:gap} and \ref{20:distance} show the change in tree size due to adding a single cut against the percentage gap closed and depth of cut respectively, for all separating cover cuts across all instances. 
Note that the points above the $y=0$ line denote an increase in the tree size while points below the line indicate a decrease in tree size. 
Contrary to expectation, $19.8\%$ of the cuts led to an increase in tree size. Moreover, the correlation of the change in tree size with the gap closed by the cut as well as the depth of cut was weak. 

For the second experiment, all separating cuts found for the instance were simultaneously added. Fig.~\ref{20:all} presents the change in tree size against the gap closed by the cuts, where every point now represents an instance. 
Even when several cuts were added simultaneously, tree sizes for 8 out of 100 instances increased.
 {However, observe that the total gap closed by cuts in all of these instances was at most 10\%.

 Fig.~\ref{20:all} suggests that if the gap closed is small (also happens when just one cut is added), it is difficult to predict the direction of change in tree size. However, after enough gap is closed, there is a clear decreasing trend, and we may expect the size of the FSB-P tree to have reduced. 
With this hypothesis in mind, we conduct the next set of experiments, where we study the change in tree size as more rounds of cuts are added.
}

\subsection{Experiments on MIPLIB} \label{sec:MIPLIB_expt}
The experiments in Section~\ref{sec:MKP} 
allowed us to completely control all sources of variability in tree size, 
ensuring that any 
increase in tree size is
an example of non-monotonicity. 
However, a shortcoming of these experiments is that these instances do not represent the diversity of problems encountered in practice. With this in mind, we next consider instances from the Benchmark Set of MIPLIB 2017 \cite{miplib2017} and cuts applied by solvers in practice.

The goal of these experiments is to understand the effects of cutting planes on FSB-P tree sizes, and how that may vary as more gap is closed. In particular, the number of rounds of cuts applied at the root node is 
varied. After each additional round of cuts, the FSB-P tree size is computed.   
The underlying hypothesis is that if enough gap is closed, the size of trees must decrease as the increased effectiveness of pruning by bounds outweighs possibly worse branching decisions.
While these experiments may be more relevant practically, we can no longer control dual degeneracy or variation in tree size due to it. We therefore run every instance with 3 different seeds. 

The experiments in this section are conducted using SCIP 8.0 \cite{BestuzhevaEtal2021OO} as the MIP solver and PySCIPOpt \cite{MaherMiltenbergerPedrosoRehfeldtSchwarzSerrano2016} as the API. All computations were done on a Linux based cluster. Whenever SCIP is called, presolve and propagation are disabled, cuts are 
allowed only at the root node,
and branching variable is selected by the \textit{vanilla full strong branching} rule while disabling strong branching side-effects. 
For every instance, the ordering of variables remains the same for a fixed seed.
Moreover, independence of the branch-and-bound trees from node selection rules is ensured by providing the optimal MIP value to the solver.
In the absence of dual degeneracy, configuring SCIP to these settings eliminates all sources of performance variability in the number of nodes to the best of our understanding, and enables us to isolate the impact of any change in branching decisions.

\paragraph{Instances.}
Due to the expensive nature of full strong branching, we first restrict the set to ``easy" instances with at most 10,000 variables and 500 integer or binary variables. We also exclude infeasible instances. This set of candidates consists of 35 instances. These instances are then solved while restricting cuts to one round of separation at the root node. Those instances not solved in 72 hours are discarded. 
In addition, 1 instance was discarded since it was solved at the root node. 
This resulted in a final set of 21 instances that were used in the experiments discussed in the rest of this section.

\paragraph{Experiments.}
In these experiments, we allow SCIP to apply its default cuts at the root node, but limit the number of rounds of cuts it applies. Every instance and seed combination is solved 11 times, while the upper limit on the rounds of cuts is increased from 0 to 10. Tree size as well as the dual bounds at the root node are saved for comparison. 
We then compute the gap closed by cuts and
change in tree size with respect to the formulation without cuts (run with the same seed), using Eq.(\ref{eq:delta_T_G}). 

\paragraph{Results.}

\begin{figure}[ht]
	\centering
	\begin{subfigure}[c]{0.32\textwidth}
		\centering
		\includegraphics[width=\textwidth]{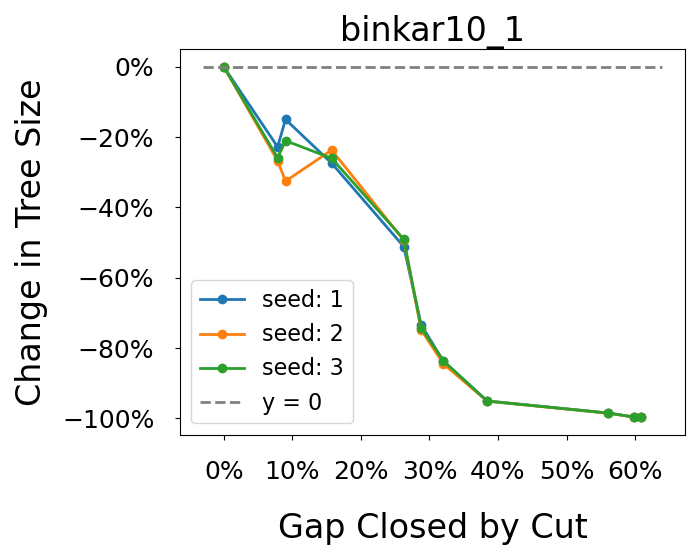}
	\end{subfigure}
	\begin{subfigure}[c]{0.32\textwidth}
		\centering
		\includegraphics[width=\textwidth]{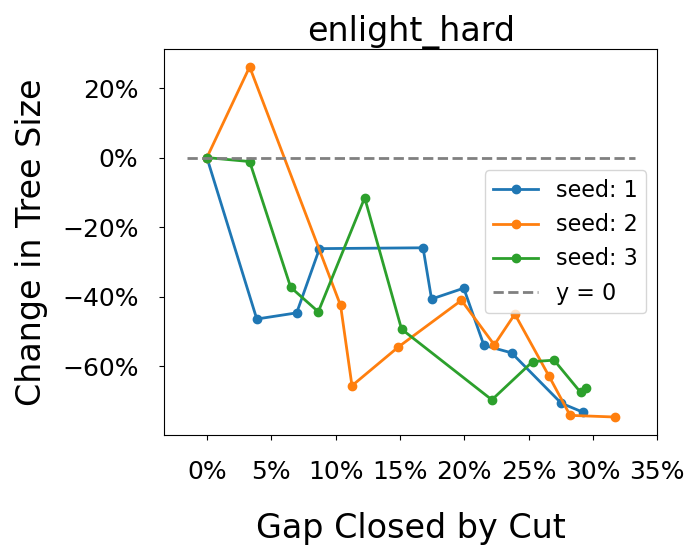}
	\end{subfigure}
        \begin{subfigure}[c]{0.32\textwidth}
		\centering
		\includegraphics[width=\textwidth]{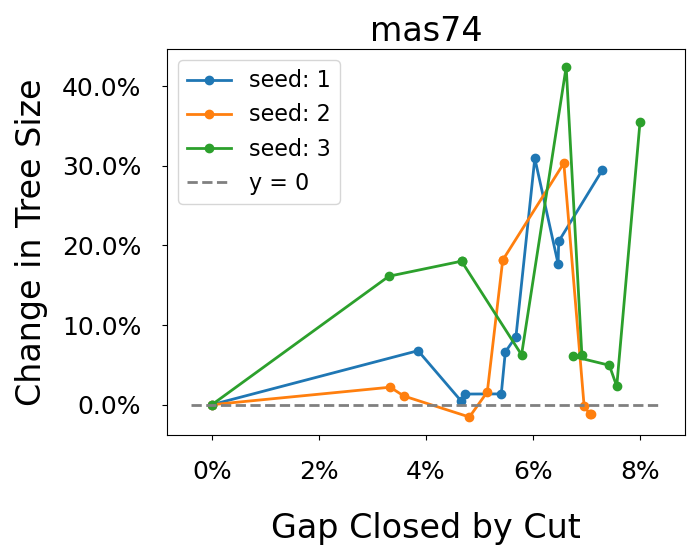}
	\end{subfigure}

        \vspace{6pt}
	\begin{subfigure}[c]{0.32\textwidth}
		\centering
		\includegraphics[width=\textwidth]{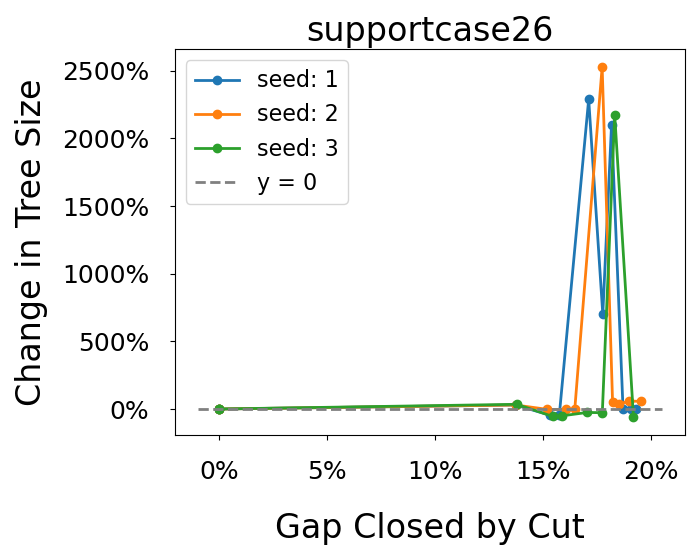}
	\end{subfigure}
        \begin{subfigure}[c]{0.32\textwidth}
		\centering
		\includegraphics[width=\textwidth]{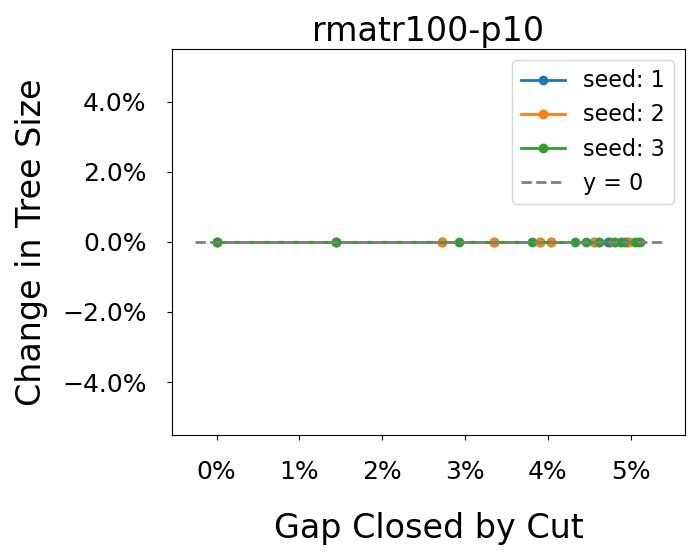}
	\end{subfigure}
	\begin{subfigure}[c]{0.32\textwidth}
		\centering
		\includegraphics[width=\textwidth]{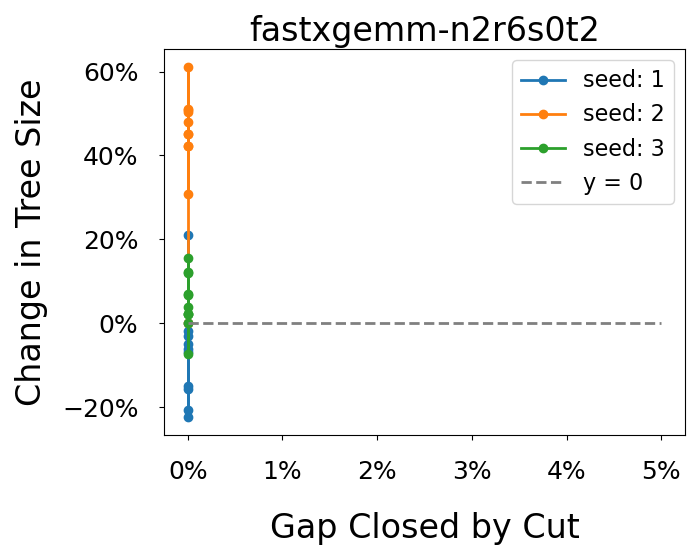}
	\end{subfigure}
	\caption{Progression of change in tree size and gap closed as the number of rounds of cuts is increased from 0 to 10.}
	\label{fig:miplib_instancewise}
 \vspace{-6pt}
\end{figure} 
For the first set of results, Fig.\ref{fig:miplib_instancewise} presents the progression of tree size and the gap closed as more rounds of cuts are added for a subset of instances. These instances have been selected to cover different patterns observed and the results for all others are included in Appendix~\ref{sec:allmiplibresults}. 

The first striking observation is that non-monotonicity was encountered in most of the instances and appears to be surprisingly common in practice.
Instances in the set showed varying sensitivity of tree size to a round of cuts and the trends in tree sizes could be categorized by the total gap closed in 10 rounds.
\begin{itemize}[parsep=6pt]
    \item When the total gap closed was large, ($\Delta G$ greater than 50\%), in line with expectation, instances binkar10\_1, mik-250-20-75-4 and n5-3 showed a steady decrease in tree size with increasing gap closed.
    \item For instances enlight\_hard and seymour1, where a moderate gap was closed ($\Delta G$ between 20\% and 50\%) the tree size had a generally decreasing trend with some intermittent rises.
    \item In most instances, the total gap closed was small ($\Delta G$ at most 20\%). These instances did not have a decreasing trend. Instances from mas, gen-ip and neos, istanbul-no-cutoff and ran14x18-disj-8 often had large spikes and drops with every round of cut. The largest jump in tree size was seen in the case of supportcase26 where, for all of the 3 seeds, the tree size increased by more than 20 times within a single round. On the other hand, rmatr100-p10 and glass-sc showed no change in tree size in spite of closing some gap.
    \item Finally, in the case of fastxgemm-n2r6s0t2, pk1, markshare\_4\_0 and mad, the gap closed is 0 even after 10 rounds of cuts, but the tree sizes changed significantly in both directions. 
    It must be noted that these instances clearly have high dual degeneracy which also contributes to variability in size.
\end{itemize}

An inference that can be drawn is that if the gap closed is small, change in tree size is difficult to predict, and often increases, possibly due to non-monotonicity. 
However, when a large enough gap is closed, a significant decrease in tree size may be expected. This is seen clearly in Fig.~\ref{fig:miplib_allinstances} where there are no data points with an increase in tree size when the gap closed exceeds 20\%. 
Note that a very similar pattern is also seen in Fig.~\ref{20:all} for the randomly generated MKP instances.

\begin{figure}[ht]
	\centering
	\begin{subfigure}[c]{\textwidth}
            \hfill
		\includegraphics[width=0.8\textwidth]{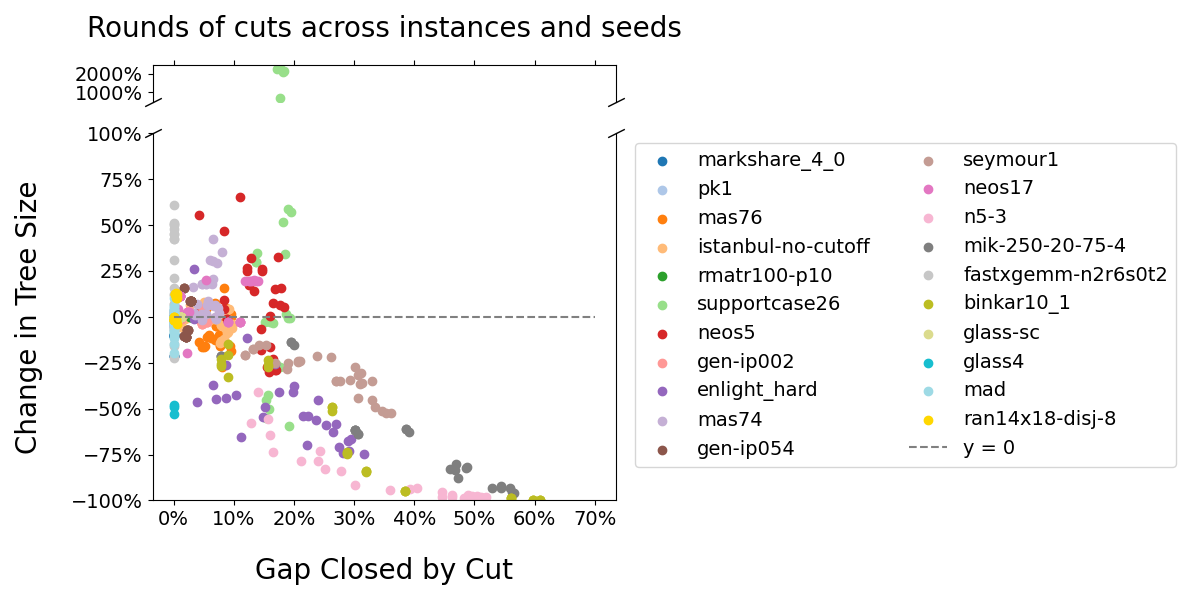}
	\end{subfigure}
	\caption{Change in tree size and the gap closed across all instances, seeds and limits on rounds of cuts.}
	\label{fig:miplib_allinstances}
 \vspace{-12pt}
\end{figure}

\section{Conclusion} \label{sec:conclusion}

The branch-and-cut method is the most efficient framework for solving general MIPs in practice. The objective of this paper is not to claim otherwise, but to highlight some of the fundamental challenges in this approach. Sometimes in practice, it is experienced that adding cuts 
increases the number of nodes.
Through this work, we show that such aberrant behaviour may not always be because of engineering reasons like 
the addition of cuts leading to a change in some default parameters or other random tie-breaking, but the problem may in fact be mathematical in nature. In particular, as shown in Section \ref{sec:theorectical_proofs}, 
standard branching rules may sometimes produce larger trees when provided with a tighter relaxation and are in general not monotonic. 
In Section \ref{sec:computational}, we show through computational experiments, that non-monotonicity is not merely a theoretical possibility, but is surprisingly prevalent in practice. When considering default cuts applied by SCIP on instances from the MIPLIB, non-monotonicity was encountered in most of the instances that were evaluated. We would like to emphasize here that these experiments are by no means an evaluation of SCIP, and the results do not reflect on its performance but only indicate inherent non-monotonicity in the branch-and-cut approach.

This work gives rise to 
many open questions in the area of the design of branching rules.
Consider a branching rule that fixes the branching decisions a priori without using local information from the linear relaxation. It is then clear that this rule is monotonic, even though it may generate large trees. On the other extreme, optimal branching ~\cite{dey2023theoretical} which considers global information is monotonic and produces small trees, but 
is \# P-hard to compute~\cite{glaser2023computing}
and is therefore impractical. It is not clear whether a good, efficient (polynomial-time complexity at each node) and monotonic rule even exists.

The results of this paper may also be viewed from the perspective of cutting-plane selection~\cite{andreello2007embedding,dey2018theoretical}.
Firstly, the depth of individual cuts is not a good measure 
of how much the tree size improves as seen in Fig.~\ref{20:distance}.
Also as we have seen, the empirical evidence from the computational experiments in Section \ref{sec:computational} indicates that if the gap closed by the group of cuts is small, we cannot be assured that adding them will result in smaller trees since their impact on branching decisions is hard to model due to possible non-monotonicity. However, when a substantial gap is closed, the size of the branch-and-bound tree decreases. 
In particular, empirical evidence suggests that crossing the 20\% gap-closed threshold via cuts could be a good rule of thumb for deciding whether the FSB-P tree size will decrease.

\paragraph{Acknowledgements.}
We would like to thank Avinash Bhardwaj for his feedback on preliminary computational results. We would also like to thank the support from AFOSR grant \# F9550-22-1-0052. 

\bibliographystyle{plain}
\bibliography{bibliography}

\begin{thebibliography}{10}

\bibitem{achterberg2007constraint}
Tobias Achterberg.
\newblock {\em Constraint integer programming}.
\newblock PhD thesis, 2007.

\bibitem{achterberg2005branching}
Tobias Achterberg, Thorsten Koch, and Alexander Martin.
\newblock Branching rules revisited.
\newblock {\em Operations Research Letters}, 33(1):42--54, 2005.

\bibitem{andreello2007embedding}
Giuseppe Andreello, Alberto Caprara, and Matteo Fischetti.
\newblock Embedding $\{$0, $1/2$$\}$-cuts in a branch-and-cut framework: A
  computational study.
\newblock {\em INFORMS Journal on Computing}, 19(2):229--238, 2007.

\bibitem{applegate1995finding}
David Applegate, Robert Bixby, Va{\v{s}}ek Chv{\'a}tal, and William Cook.
\newblock {\em Finding cuts in the TSP (A preliminary report)}, volume~95.
\newblock Citeseer, 1995.

\bibitem{balas1996mixed}
Egon Balas, Sebasti{\'a}n Ceria, and G{\'e}rard Cornu{\'e}jols.
\newblock Mixed 0-1 programming by lift-and-project in a branch-and-cut
  framework.
\newblock {\em Management Science}, 42(9):1229--1246, 1996.

\bibitem{balas1996gomory}
Egon Balas, Sebastian Ceria, G{\'e}rard Cornu{\'e}jols, and N~Natraj.
\newblock Gomory cuts revisited.
\newblock {\em Operations Research Letters}, 19(1):1--9, 1996.

\bibitem{basu2021complexity}
Amitabh Basu, Michele Conforti, Marco Di~Summa, and Hongyi Jiang.
\newblock Complexity of branch-and-bound and cutting planes in mixed-integer
  optimization-ii.
\newblock In {\em International Conference on Integer Programming and
  Combinatorial Optimization}, pages 383--398. Springer, 2021.

\bibitem{basu2023complexity}
Amitabh Basu, Michele Conforti, Marco Di~Summa, and Hongyi Jiang.
\newblock Complexity of branch-and-bound and cutting planes in mixed-integer
  optimization.
\newblock {\em Mathematical Programming}, 198(1):787--810, 2023.

\bibitem{BestuzhevaEtal2021OO}
Ksenia Bestuzheva, Mathieu Besan{\c{c}}on, Wei-Kun Chen, Antonia Chmiela, Tim
  Donkiewicz, Jasper van Doornmalen, Leon Eifler, Oliver Gaul, Gerald Gamrath,
  Ambros Gleixner, Leona Gottwald, Christoph Graczyk, Katrin Halbig, Alexander
  Hoen, Christopher Hojny, Rolf van~der Hulst, Thorsten Koch, Marco
  L{\"u}bbecke, Stephen~J. Maher, Frederic Matter, Erik M{\"u}hmer, Benjamin
  M{\"u}ller, Marc~E. Pfetsch, Daniel Rehfeldt, Steffan Schlein, Franziska
  Schl{\"o}sser, Felipe Serrano, Yuji Shinano, Boro Sofranac, Mark Turner,
  Stefan Vigerske, Fabian Wegscheider, Philipp Wellner, Dieter Weninger, and
  Jakob Witzig.
\newblock {The SCIP Optimization Suite 8.0}.
\newblock Technical report, Optimization Online, December 2021.

\bibitem{bodur2017cutting}
Merve Bodur, Sanjeeb Dash, and Oktay G{\"u}nl{\"u}k.
\newblock Cutting planes from extended lp formulations.
\newblock {\em Mathematical Programming}, 161:159--192, 2017.

\bibitem{cook2010fifty}
William Cook.
\newblock Fifty-plus years of combinatorial integer programming.
\newblock {\em 50 Years of Integer Programming 1958-2008: From the Early Years
  to the State-of-the-Art}, pages 387--430, 2010.

\bibitem{crowder1983solving}
Harlan Crowder, Ellis~L Johnson, and Manfred Padberg.
\newblock Solving large-scale zero-one linear programming problems.
\newblock {\em Operations Research}, 31(5):803--834, 1983.

\bibitem{dadush2020complexity}
Daniel Dadush and Samarth Tiwari.
\newblock On the complexity of branching proofs.
\newblock In {\em Proceedings of the 35th Computational Complexity Conference},
  pages 1--35, 2020.

\bibitem{dantzig1954solution}
George Dantzig, Ray Fulkerson, and Selmer Johnson.
\newblock Solution of a large-scale traveling-salesman problem.
\newblock {\em Journal of the operations research society of America},
  2(4):393--410, 1954.

\bibitem{dey2021branch}
Santanu~S Dey, Yatharth Dubey, and Marco Molinaro.
\newblock Branch-and-bound solves random binary ips in polytime.
\newblock In {\em Proceedings of the 2021 ACM-SIAM Symposium on Discrete
  Algorithms (SODA)}, pages 579--591. SIAM, 2021.

\bibitem{dey2023theoretical}
Santanu~S Dey, Yatharth Dubey, Marco Molinaro, and Prachi Shah.
\newblock A theoretical and computational analysis of full strong-branching.
\newblock {\em Mathematical Programming}, pages 1--34, 2023.

\bibitem{dey2018theoretical}
Santanu~S Dey and Marco Molinaro.
\newblock Theoretical challenges towards cutting-plane selection.
\newblock {\em Mathematical Programming}, 170:237--266, 2018.

\bibitem{gamrath2018measuring}
Gerald Gamrath and Christoph Schubert.
\newblock Measuring the impact of branching rules for mixed-integer
  programming.
\newblock In {\em Operations Research Proceedings 2017: Selected Papers of the
  Annual International Conference of the German Operations Research Society
  (GOR), Freie Universi{\"a}t Berlin, Germany, September 6-8, 2017}, pages
  165--170. Springer, 2018.

\bibitem{glaser2023computing}
Max Gl{\"a}ser and Marc~E Pfetsch.
\newblock On computing small variable disjunction branch-and-bound trees.
\newblock {\em Mathematical Programming}, pages 1--29, 2023.

\bibitem{miplib2017}
Ambros Gleixner, Gregor Hendel, Gerald Gamrath, Tobias Achterberg, Michael
  Bastubbe, Timo Berthold, Philipp~M. Christophel, Kati Jarck, Thorsten Koch,
  Jeff Linderoth, Marco L\"ubbecke, Hans~D. Mittelmann, Derya Ozyurt, Ted~K.
  Ralphs, Domenico Salvagnin, and Yuji Shinano.
\newblock {MIPLIB 2017: Data-Driven Compilation of the 6th Mixed-Integer
  Programming Library}.
\newblock {\em Mathematical Programming Computation}, 2021.

\bibitem{gomoryoutline}
Ralph~E Gomory.
\newblock Outline of an algorithm for integer solutions to linear programs.
\newblock {\em Bull. Amer. Math. Soc.}, 1958.

\bibitem{kazachkov2023abstract}
Aleksandr~M Kazachkov, Pierre Le~Bodic, and Sriram Sankaranarayanan.
\newblock An abstract model for branch and cut.
\newblock {\em Mathematical Programming}, pages 1--28, 2023.

\bibitem{land60a}
Ailsa~H. Land and Alison~G. Doig.
\newblock An automatic method of solving discrete programming problems.
\newblock {\em Econometrica}, 28(3):497--520, 1960.

\bibitem{le2017abstract}
Pierre Le~Bodic and George Nemhauser.
\newblock An abstract model for branching and its application to mixed integer
  programming.
\newblock {\em Mathematical Programming}, 166(1-2):369--405, 2017.

\bibitem{linderoth1999computational}
Jeff~T Linderoth and Martin~WP Savelsbergh.
\newblock A computational study of search strategies for mixed integer
  programming.
\newblock {\em INFORMS Journal on Computing}, 11(2):173--187, 1999.

\bibitem{lodi2010mixed}
Andrea Lodi.
\newblock Mixed integer programming computation.
\newblock {\em 50 Years of Integer Programming 1958-2008: From the Early Years
  to the State-of-the-Art}, pages 619--645, 2010.

\bibitem{MaherMiltenbergerPedrosoRehfeldtSchwarzSerrano2016}
Stephen Maher, Matthias Miltenberger, Jo{\~{a}}o~Pedro Pedroso, Daniel
  Rehfeldt, Robert Schwarz, and Felipe Serrano.
\newblock {PySCIPOpt}: Mathematical programming in python with the {SCIP}
  optimization suite.
\newblock In {\em Mathematical Software {\textendash} {ICMS} 2016}, pages
  301--307. Springer International Publishing, 2016.

\bibitem{padberg1991branch}
Manfred Padberg and Giovanni Rinaldi.
\newblock A branch-and-cut algorithm for the resolution of large-scale
  symmetric traveling salesman problems.
\newblock {\em SIAM review}, 33(1):60--100, 1991.

\bibitem{van1987solving}
Tony~J Van~Roy and Laurence~A Wolsey.
\newblock Solving mixed integer programming problems using automatic
  reformulation.
\newblock {\em Operations Research}, 35(1):45--57, 1987.

\end{thebibliography}
\newpage

\appendix

\section{MIPLIB Results} \label{sec:allmiplibresults}

The plots for all instances considered in Section. \ref{sec:MIPLIB_expt} are presented in Fig. \ref{fig:miplib_appendix_instancewise}.
\begin{figure}[H]
	\centering
	\begin{subfigure}[c]{0.32\textwidth}
		\centering
		\includegraphics[width=\textwidth]{images/normalized_binkar10_1.png}
	\end{subfigure}
	\begin{subfigure}[c]{0.32\textwidth}
		\centering
		\includegraphics[width=\textwidth]{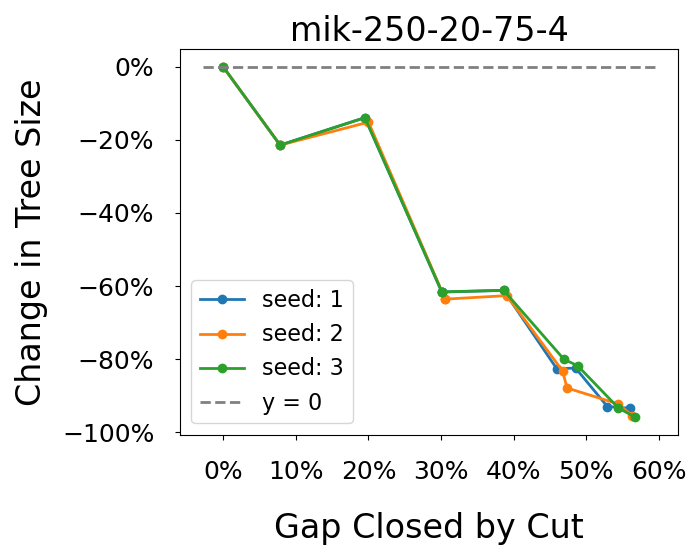}
	\end{subfigure}
        \begin{subfigure}[c]{0.32\textwidth}
		\centering
		\includegraphics[width=\textwidth]{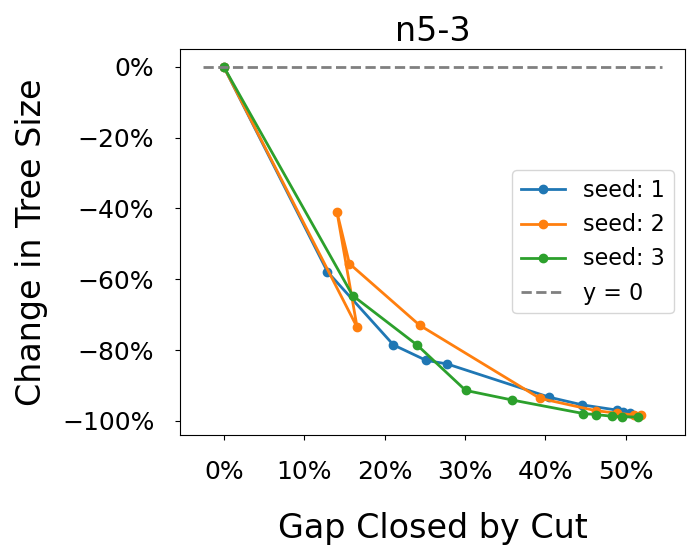}
	\end{subfigure}
	\begin{subfigure}[c]{0.32\textwidth}
		\centering
		\includegraphics[width=\textwidth]{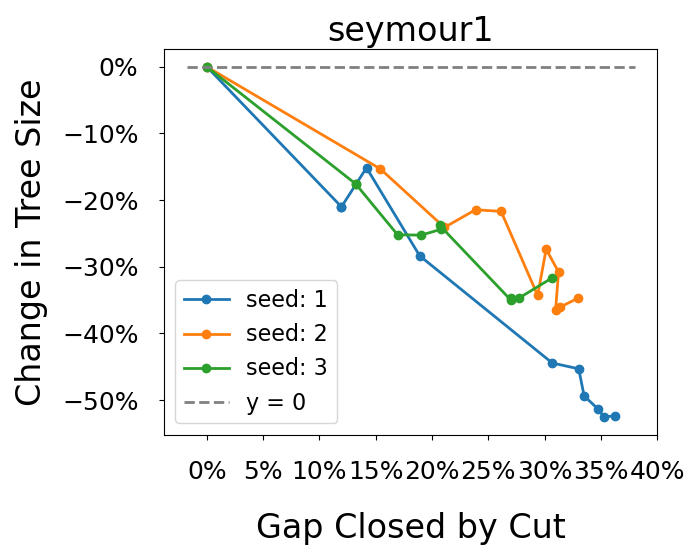}
	\end{subfigure}
        \begin{subfigure}[c]{0.32\textwidth}
		\centering
		\includegraphics[width=\textwidth]{images/normalized_enlight_hard.png}
	\end{subfigure}
	\begin{subfigure}[c]{0.32\textwidth}
		\centering
		\includegraphics[width=\textwidth]{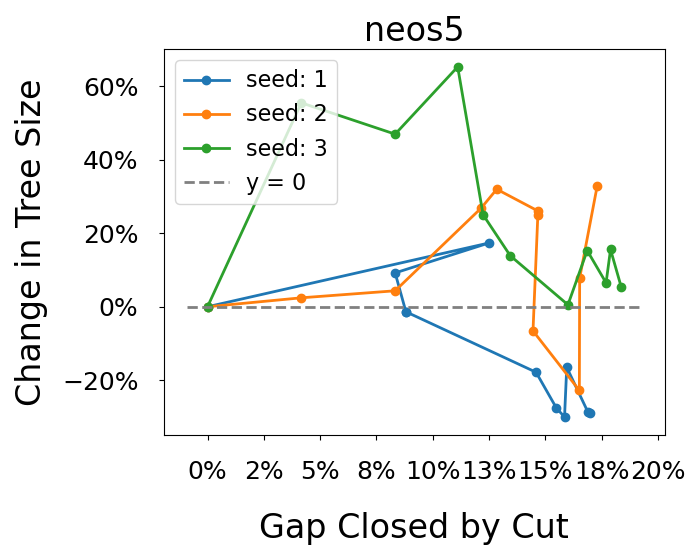}
	\end{subfigure}
	\begin{subfigure}[c]{0.32\textwidth}
		\centering
		\includegraphics[width=\textwidth]{images/normalized_supportcase26.png}
	\end{subfigure}
	\begin{subfigure}[c]{0.32\textwidth}
		\centering
		\includegraphics[width=\textwidth]{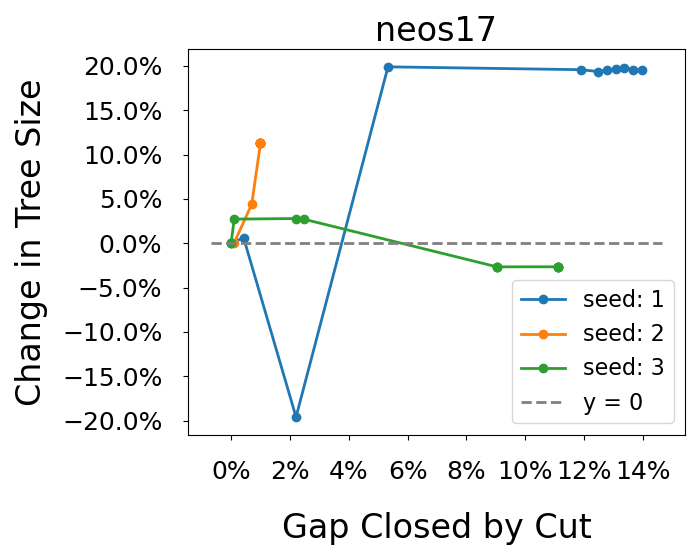}
	\end{subfigure}
        \begin{subfigure}[c]{0.32\textwidth}
		\centering
		\includegraphics[width=\textwidth]{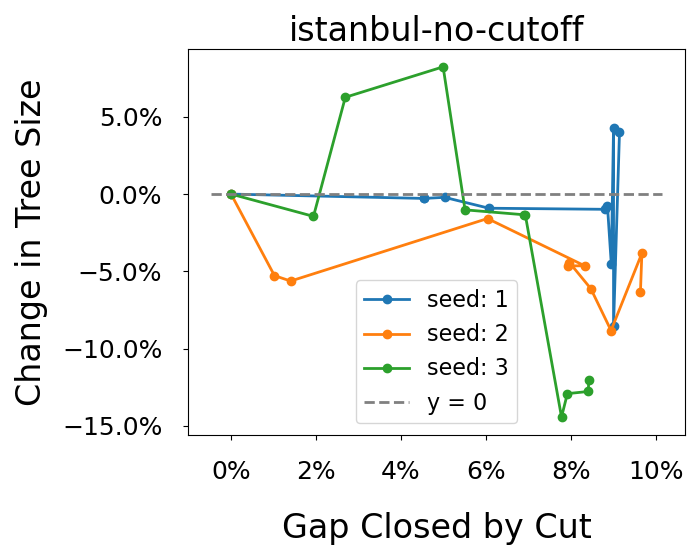}
	\end{subfigure}
        \begin{subfigure}[c]{0.32\textwidth}
		\centering
		\includegraphics[width=\textwidth]{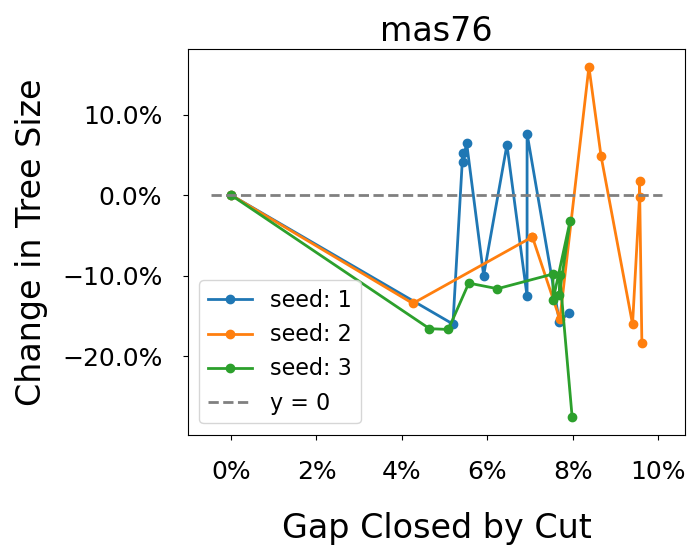}
	\end{subfigure}
        \begin{subfigure}[c]{0.32\textwidth}
		\centering
		\includegraphics[width=\textwidth]{images/normalized_mas74.png}
	\end{subfigure}
        \begin{subfigure}[c]{0.32\textwidth}
		\centering
		\includegraphics[width=\textwidth]{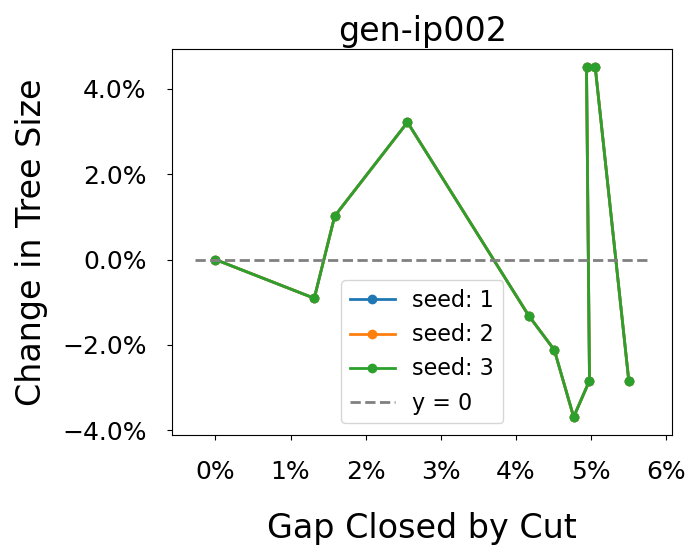}
	\end{subfigure}
	\caption{Progression of change in tree size and gap closed as number of rounds of cuts are increased from 0 to 10. \emph{(cont.)}}
\end{figure} 
\begin{figure}[H]
\ContinuedFloat
	\centering
	\begin{subfigure}[c]{0.32\textwidth}
		\centering
		\includegraphics[width=\textwidth]{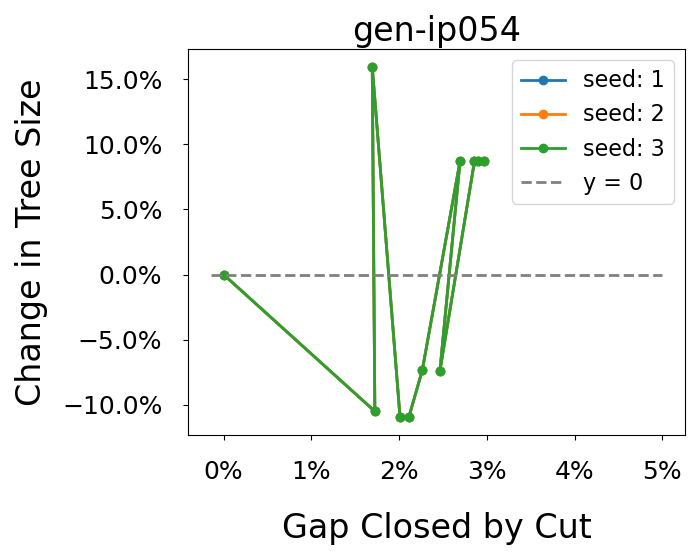}
	\end{subfigure}
        \begin{subfigure}[c]{0.32\textwidth}
		\centering
		\includegraphics[width=\textwidth]{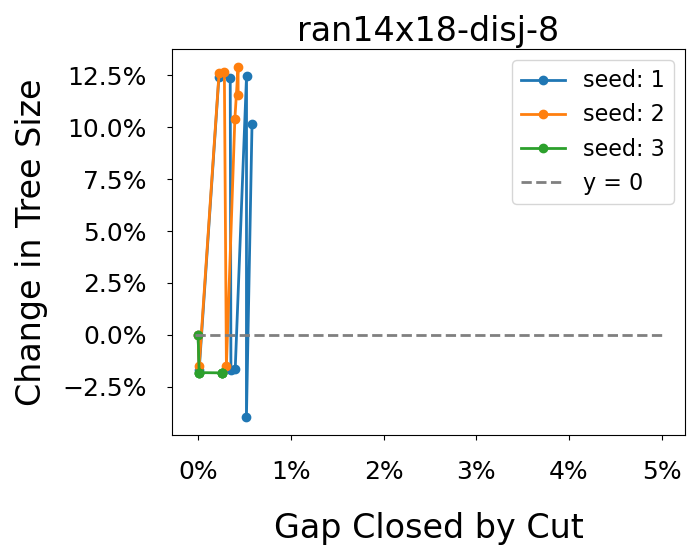}
	\end{subfigure}
        \begin{subfigure}[c]{0.32\textwidth}
		\centering
		\includegraphics[width=\textwidth]{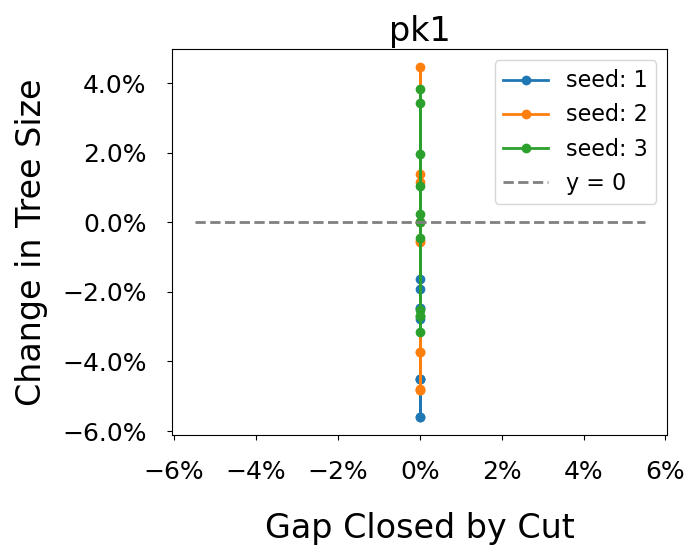}
	\end{subfigure}
	\begin{subfigure}[c]{0.32\textwidth}
		\centering
		\includegraphics[width=\textwidth]{images/normalized_fastxgemm-n2r6s0t2.png}
	\end{subfigure}
        \begin{subfigure}[c]{0.32\textwidth}
		\centering
		\includegraphics[width=\textwidth]{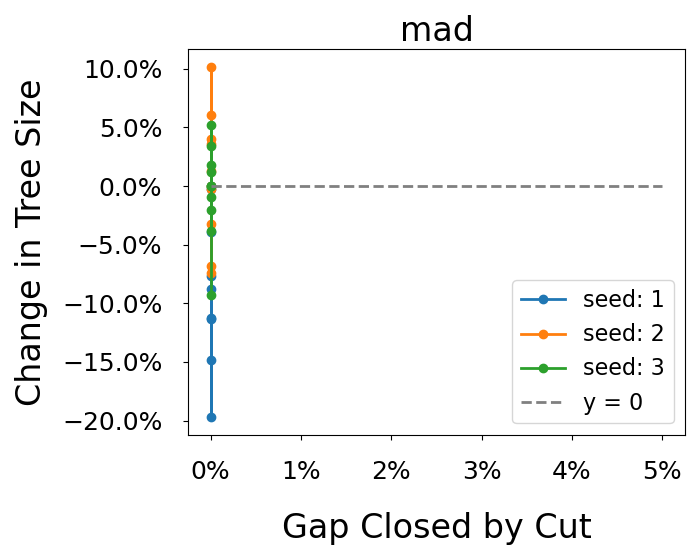}
	\end{subfigure}
	\begin{subfigure}[c]{0.32\textwidth}
		\centering
		\includegraphics[width=\textwidth]{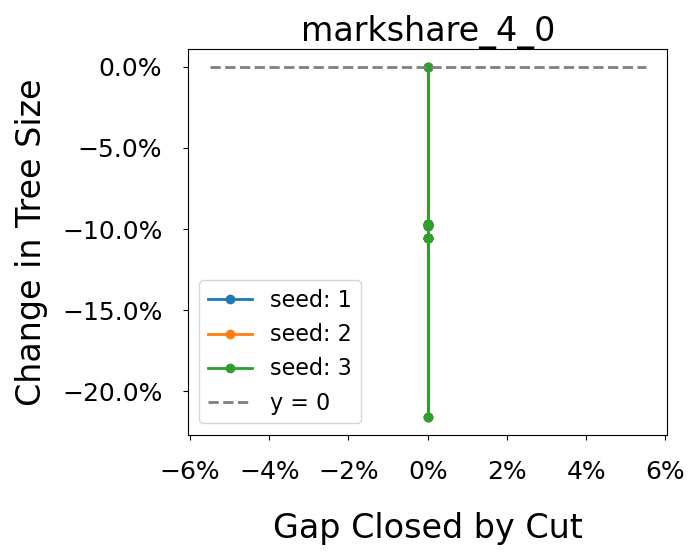}
        \end{subfigure}
        \begin{subfigure}[c]{0.32\textwidth}
		\centering
		\includegraphics[width=\textwidth]{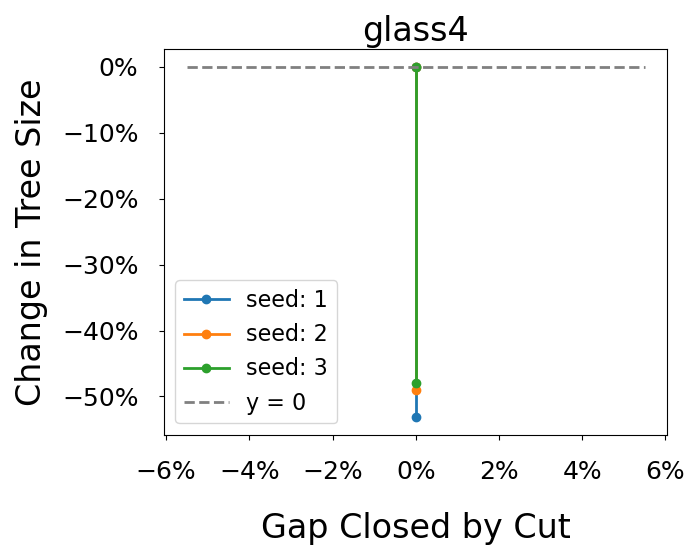}
	\end{subfigure}
        \begin{subfigure}[c]{0.32\textwidth}
		\centering
		\includegraphics[width=\textwidth]{images/normalized_rmatr100-p10.png}
	\end{subfigure}
	\begin{subfigure}[c]{0.32\textwidth}
		\centering
		\includegraphics[width=\textwidth]{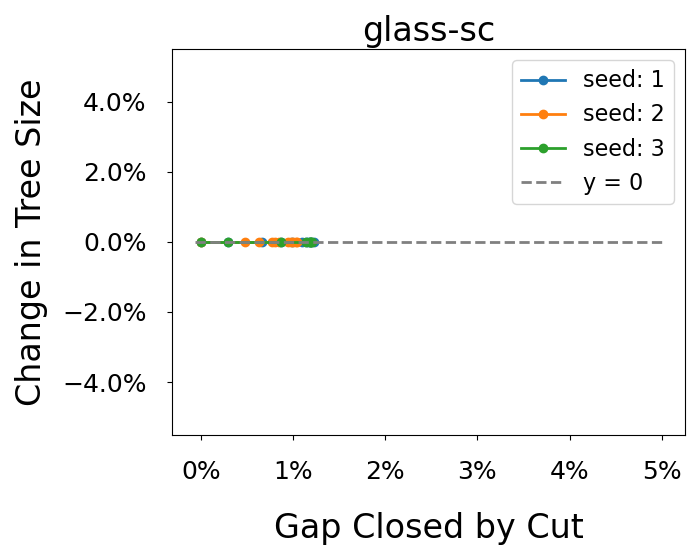}
	\end{subfigure}
	\caption{Progression of change in tree size and gap closed as number of rounds of cuts are increased from 0 to 10.}
	\label{fig:miplib_appendix_instancewise}
\end{figure} 


	


\end{document}